\documentclass[12pt]{article}
\usepackage{amsmath,amssymb,amsthm}
\usepackage{tikz}

\def\picAAA{
\begin{tikzpicture}[scale=0.75]
\tikzstyle{knode}=[circle,draw=black,thick,inner sep=2pt]
  \node (n1) at (22.5:1.5cm) [knode] {};
  \node (n2) at (67.5:1.5cm) [knode] {};
  \node (n3) at (112.5:1.5cm) [knode] {};
  \node (n4) at (157.5:1.5cm) [knode] {};
  \node (n5) at (202.5:1.5cm)  [knode] {};
  \node (n6) at (247.5:1.5cm) [knode] {};
  \node (n7) at (292.5:1.5cm) [knode] {}; 
  \node (n8) at (337.5:1.5cm) [knode] {};
  
  \path (n1)+(0:1cm) node (n9) [knode] {}
  (n1)+(45:1cm) node (n10) [knode] {}
  (n2)+(45:1cm) node (n11) [knode] {}
  (n2)+(90:1cm) node (n12) [knode] {}
  (n3)+(90:1cm) node (n13) [knode] {}
  (n3)+(135:1cm) node (n14) [knode] {}
  (n4)+(135:1cm) node (n15) [knode] {}
  (n4)+(180:1cm) node (n16) [knode] {}
  (n5)+(180:1cm) node (n17) [knode] {}
  (n5)+(-135:1cm) node (n18) [knode] {}
  (n6)+(-135:1cm) node (n19) [knode] {}
  (n6)+(-90:1cm) node (n20) [knode] {}
  (n7)+(-90:1cm) node (n21) [knode] {}
  (n7)+(-45:1cm) node (n22) [knode] {}
  (n8)+(-45:1cm) node (n23) [knode] {}
  (n8)+(0:1cm) node (n24) [knode] {}
  ;
  
  \draw (n1)--(n2)--(n3)--(n4)--(n5)--(n6)--(n7)--(n8)--(n1);
  \foreach \from/\to in {n1/n9,n1/n10,n2/n11,n2/n12,n3/n13,n3/n14,n4/n15,n4/n16,n5/n17,n5/n18,n6/n19,n6/n20,n7/n21,n7/n22,n8/n23,n8/n24} 
  \draw (\from)--(\to);
  
  \draw (n9)--(n24)
  (n14)--(n15)
  (n22)--(n23); 
  
\end{tikzpicture}
\hfil
 \begin{tikzpicture}[scale=0.75]
\tikzstyle{knode}=[circle,draw=black,thick,inner sep=2pt]
  \node (n1) at (22.5:1.5cm) [knode] {};
  \node (n2) at (67.5:1.5cm) [knode] {};
  \node (n3) at (112.5:1.5cm) [knode] {};
  \node (n4) at (157.5:1.5cm) [knode] {};
  \node (n5) at (202.5:1.5cm)  [knode] {};
  \node (n6) at (247.5:1.5cm) [knode] {};
  \node (n7) at (292.5:1.5cm) [knode] {}; 
  \node (n8) at (337.5:1.5cm) [knode] {};
  
  \path (n1)+(0:1cm) node (n9) [knode] {}
  (n1)+(45:1cm) node (n10) [knode] {}
  (n2)+(45:1cm) node (n11) [knode] {}
  (n2)+(90:1cm) node (n12) [knode] {}
  (n3)+(90:1cm) node (n13) [knode] {}
  (n3)+(135:1cm) node (n14) [knode] {}
  (n4)+(135:1cm) node (n15) [knode] {}
  (n4)+(180:1cm) node (n16) [knode] {}
  (n5)+(180:1cm) node (n17) [knode] {}
  (n5)+(-135:1cm) node (n18) [knode] {}
  (n6)+(-135:1cm) node (n19) [knode] {}
  (n6)+(-90:1cm) node (n20) [knode] {}
  (n7)+(-90:1cm) node (n21) [knode] {}
  (n7)+(-45:1cm) node (n22) [knode] {}
  (n8)+(-45:1cm) node (n23) [knode] {}
  (n8)+(0:1cm) node (n24) [knode] {}
  ;
  
  \draw (n1)--(n2)--(n3)--(n4)--(n5)--(n6)--(n7)--(n8)--(n1);
  \foreach \from/\to in {n1/n9,n1/n10,n2/n11,n2/n12,n3/n13,n3/n14,n4/n15,n4/n16,n5/n17,n5/n18,n6/n19,n6/n20,n7/n21,n7/n22,n8/n23,n8/n24} 
  \draw (\from)--(\to);
  
  \draw (n10)--(n11)
  (n12)--(n13)
  (n16)--(n17)
  (n18)--(n19)
  (n20)--(n21)
  ; 
  
\end{tikzpicture}}

\newcommand{\posnode}[1]{\node[inner sep=-1pt,circle,draw=black,fill=black] #1 {\begin{tikzpicture}\draw[ultra thick, color=white] (0,0)--(0.2,0);\draw[ultra thick, color=white] (0.1,-0.1)--(0.1,0.1);\end{tikzpicture}};}
\newcommand{\negnode}[1]{\node[inner sep=0pt,circle,draw=black,fill=black] #1 {\begin{tikzpicture}\draw[ultra thick, color=white] (0,0)--(0.2,0);\end{tikzpicture}};}

\def\picBBB{\begin{tikzpicture}[scale=.2]
\tikzstyle{knode}=[circle,draw=black,thick,inner sep=2pt]
\posnode{(n1) at (-3,0)}
\negnode{(n2) at (3,0)}
\node (n3) at (-3,6) [knode] {};
\node (n4) at (3,6) [knode] {};
\node (n5) at (0,7) {$k^2$};
\node (n6) at (-3.75,3) {$k$};
\node (n7) at (3.75,3) {$k$};
\node (n8) at (0,-1) {$1$};
\draw (n1)--(n2)--(n4)--(n3)--(n1);

\posnode{(n11) at (-13,0)}
\negnode{(n12) at (-7,0)}
\node (n13) at (-13,6) [knode] {};
\node (n14) at (-7,6) [knode] {};
\node (n16) at (-13.75,3) {$k$};
\node (n17) at (-6.25,3) {$k$};
\node (n18) at (-10,-1) {$1$};
\draw (n13)--(n11)--(n12)--(n14);

\posnode{(n21) at (7,0)}
\negnode{(n22) at (13,0)}
\node (n28) at (10,-1) {$k+1$};
\draw (n21)--(n22);
\end{tikzpicture}}

\def\picCCC{\begin{tikzpicture}[scale=0.5]
\tikzstyle{knode}=[circle,draw=black,thick,inner sep = 2pt]
\posnode{(n1) at (-1cm,0cm)}
\negnode{(n2) at (1cm,0cm)}
\node (n3) at (-1cm, 2cm) [knode] {};
\node (n4) at (1cm, 2cm) [knode] {};
\draw (n1) -- (n2); 
\draw[dashed] (n1)--(n3)
(n2) -- (n4);
\end{tikzpicture}
\hfil
\begin{tikzpicture}[scale=0.5]
\tikzstyle{knode}=[circle,draw=black,thick,inner sep = 2pt]
\posnode{(n1) at (-1cm,0cm)}
\negnode{(n2) at (1cm,0cm)}
\draw (n1) -- (n2); 
\end{tikzpicture}
}

\def\picDDD{\begin{tikzpicture}[scale=0.5]
\tikzstyle{knode}=[circle,draw=black,thick, inner sep = 2pt]
\posnode{(n1) at (-1cm,0cm)}
\negnode{(n2) at (1cm,0cm)}
\node (n3) at (-1cm, 2cm) [knode] {};
\node (n4) at (1cm, 2cm) [knode] {};
\draw (n1) -- (n2); 
\draw[dashed] (n1)--(n3)
(n2) -- (n4)
(n3)--(n4);
\end{tikzpicture}
\hfil 
\begin{tikzpicture}[scale=0.5]
\tikzstyle{knode}=[circle,draw=black,thick, inner sep = 2pt]
\posnode{(n1) at (-1cm,0cm)}
\negnode{(n2) at (1cm,0cm)}
\node (n3) at (-1cm, 2cm) [knode] {};
\node (n4) at (1cm, 2cm) [knode] {};
\draw (n1) -- (n2) (n3) -- (n4);
\draw[dashed] (n1)--(n3) (n2) -- (n4);
\end{tikzpicture}
\hfil 
\begin{tikzpicture}[scale=0.5]
\tikzstyle{knode}=[circle,draw=black,thick,inner sep = 2pt]
\posnode{(n1) at (-1cm,0cm)}
\negnode{(n2) at (1cm,0cm)}
\node (n3) at (-1cm, 2cm) [knode] {};
\node (n4) at (1cm, 2cm) [knode] {};
\draw (n1) -- (n3)
(n3)--(n4)
(n4)--(n2); 
\draw[dashed] (n1)--(n2);
\end{tikzpicture}}

\def\tableP{\begin{table}[hftb]
\centering
\caption{Decompositions of the {\tt P} module}
\label{table:P}
\begin{tabular}{|c|c|c|}\hline signed vertex  & local decomposition & contribution \\ \hline 
neither &
\begin{tikzpicture}[scale=.4] 
\tikzstyle{knode}=[circle,draw=black,thick,inner sep=2pt]
\posnode{(n1) at (-1cm,0cm)}
\negnode{(n2) at (1cm,0cm)}
\node (n3) at (-1cm, 2cm) [knode] {};
\node (n4) at (1cm, 2cm) [knode] {};
\draw[dashed] (n1) -- (n2) (n1) -- (n3) (n2)-- (n4);
\end{tikzpicture} 
& 
$\displaystyle1$  \\ \hline 
$+$ & 
\begin{tikzpicture}[scale=.4]
\tikzstyle{knode}=[circle,draw=black,thick,inner sep=2pt]
\posnode{(n1) at (-1cm,0cm)}
\negnode{(n2) at (1cm,0cm)}
\node (n3) at (-1cm, 2cm) [knode] {};
\node (n4) at (1cm, 2cm) [knode] {};
\draw (n1)--(n3);
\draw[dashed] (n1) -- (n2) (n2) -- (n4);
\end{tikzpicture}  
& 
$\displaystyle \frac{-k}{(t-1)^2(2k+2)}$ \\ \hline 
$-$ & 
\begin{tikzpicture}[scale=.4] 
\tikzstyle{knode}=[circle,draw=black,thick,inner sep=2pt]
\posnode{(n1) at (-1cm,0cm)}
\negnode{(n2) at (1cm,0cm)}
\node (n3) at (-1cm, 2cm) [knode] {};
\node (n4) at (1cm, 2cm) [knode] {};
\draw (n2)--(n4);
\draw[dashed] (n1) -- (n2) (n1) -- (n3);
\end{tikzpicture} 
& 
$ \displaystyle \frac{-k}{(t-1)^2(2k+2)}$ \\ \hline 
${+}/{-}$ &
\begin{tikzpicture}[scale=.4]
\tikzstyle{knode}=[circle,draw=black,thick,inner sep=2pt]
\posnode{(n1) at (-1cm,0cm)}
\negnode{(n2) at (1cm,0cm)}
\node (n3) at (-1cm, 2cm) [knode] {};
\node (n4) at (1cm, 2cm) [knode] {};
\draw (n1)--(n2);
\draw[dashed] (n1)--(n3) (n2) -- (n4);
\end{tikzpicture}
\hspace{.2 cm}
\begin{tikzpicture}[scale=.4]
\tikzstyle{knode}=[circle,draw=black,thick,inner sep=2pt]
\posnode{(n1) at (-1cm,0cm)}
\negnode{(n2) at (1cm,0cm)}
\node (n3) at (-1cm, 2cm) [knode] {};
\node (n4) at (1cm, 2cm) [knode] {};
\draw (n1)--(n3) (n2)--(n4);
\draw[dashed] (n1) -- (n2);
\end{tikzpicture}
  & $\displaystyle \frac{k^2-(t-1)^2}{(t-1)^4(2k+2)^2}$  \\\hline 
\end{tabular} 
\end{table}}

\def\tableC{\begin{table}[hftb]
\centering
\caption{Decompositions of the {\tt C} module}
\label{table:C}
\begin{tabular}{|c|c|c|}\hline signed vertex & local decomposition & contribution \\\hline 
neither & 
\begin{tikzpicture}[scale=.4] 
\tikzstyle{knode}=[circle,draw=black,thick,inner sep=2pt]
\posnode{(n1) at (-1cm,0cm)}
\negnode{(n2) at (1cm,0cm)}
\node (n3) at (-1cm, 2cm) [knode] {};
\node (n4) at (1cm, 2cm) [knode] {};
\draw[dashed] (n1) -- (n2)
(n1) -- (n3)
(n3) -- (n4)
(n2)-- (n4);
\end{tikzpicture}
\hspace{.2 cm} 
\begin{tikzpicture}[scale=.4]
\tikzstyle{knode}=[circle,draw=black,thick,inner sep=2pt]
\posnode{(n1) at (-1cm,0cm)}
\negnode{(n2) at (1cm,0cm)}
\node (n3) at (-1cm, 2cm) [knode] {};
\node (n4) at (1cm, 2cm) [knode] {};
\draw[dashed] (n1)--(n3)
(n2) -- (n4)
(n1) -- (n2);
\draw (n3)--(n4);
\end{tikzpicture}& $\displaystyle 1-\frac{k^2}{(t-1)^2(k+1)^2}$ \\\hline 
$+$ &   
\begin{tikzpicture}[scale=.4]
\tikzstyle{knode}=[circle,draw=black,thick,inner sep=2pt]
\posnode{(n1) at (-1cm,0cm)}
\negnode{(n2) at (1cm,0cm)}
\node (n3) at (-1cm, 2cm) [knode] {};
\node (n4) at (1cm, 2cm) [knode] {};
\draw (n1)--(n3);
\draw[dashed] (n1) -- (n2)
(n3)--(n4)
(n2) -- (n4);
\end{tikzpicture} & $\displaystyle \frac{-k}{2(t-1)^2(k+1)^2}$ \\\hline 
$-$ & \begin{tikzpicture}[scale=.4] 
\tikzstyle{knode}=[circle,draw=black,thick,inner sep=2pt]
\posnode{(n1) at (-1cm,0cm)}
\negnode{(n2) at (1cm,0cm)}
\node (n3) at (-1cm, 2cm) [knode] {};
\node (n4) at (1cm, 2cm) [knode] {};
\draw (n2)--(n4);
\draw[dashed] (n1) -- (n2)
(n3)--(n4)
(n1) -- (n3);
\end{tikzpicture}  & $\displaystyle \frac{-k}{2(t-1)^2(k+1)^2}$ \\\hline 
${+}/{-}$ & \begin{tikzpicture}[scale=.4]
\tikzstyle{knode}=[circle,draw=black,thick,inner sep=2pt]
\posnode{(n1) at (-1cm,0cm)}
\negnode{(n2) at (1cm,0cm)}
\node (n3) at (-1cm, 2cm) [knode] {};
\node (n4) at (1cm, 2cm) [knode] {};
\draw (n1)--(n2);
\draw[dashed] (n1)--(n3)
(n3) -- (n4)
(n2) -- (n4);
\end{tikzpicture}
\hspace{.2 cm}
\begin{tikzpicture}[scale=.4]
\tikzstyle{knode}=[circle,draw=black,thick,inner sep=2pt]
\posnode{(n1) at (-1cm,0cm)}
\negnode{(n2) at (1cm,0cm)}
\node (n3) at (-1cm, 2cm) [knode] {};
\node (n4) at (1cm, 2cm) [knode] {};
\draw (n1)--(n3)  (n2)--(n4);
\draw[dashed] (n1) -- (n2)
(n3) -- (n4);
\end{tikzpicture}
\hspace{.2 cm}
\begin{tikzpicture}[scale=.4]
\tikzstyle{knode}=[circle,draw=black,thick,inner sep=2pt]
\posnode{(n1) at (-1cm,0cm)}
\negnode{(n2) at (1cm,0cm)}
\node (n3) at (-1cm, 2cm) [knode] {};
\node (n4) at (1cm, 2cm) [knode] {};
\draw[dashed] (n1)--(n3)
(n2) -- (n4); 
\draw (n3)--(n4)  (n1)--(n2);
\end{tikzpicture}
\hspace{.2 cm}
\begin{tikzpicture}[scale=.4]
\tikzstyle{knode}=[circle,draw=black,thick,inner sep=2pt]
\posnode{(n1) at (-1cm,0cm)}
\negnode{(n2) at (1cm,0cm)}
\node (n3) at (-1cm, 2cm) [knode] {};
\node (n4) at (1cm, 2cm) [knode] {};
\draw(n1)--(n3)
(n2) -- (n4) 
(n1) -- (n2)
(n3) -- (n4);
\end{tikzpicture} & $\displaystyle \frac{-1}{4(t-1)^2(k+1)^2}$ \\\hline \end{tabular} 
\end{table}}

\def\tableE{\begin{table}[hftb]
\centering
\caption{Decompositions of the {\tt E} module}
\label{table:E}
\begin{tabular}{|c|c|c|}\hline used & local decomposition & factor \\\hline neither & \begin{tikzpicture}[scale=.4]
\tikzstyle{knode}=[circle,draw=black,thick,inner sep=2pt]
\posnode{(n1) at (-1cm,0cm)}
\negnode{(n2) at (1cm,0cm)}
\draw[dashed] (n1) -- (n2);
\end{tikzpicture} & $1$  \\\hline $+$ &  & 0 \\\hline $-$ &  & 0 \\\hline ${+}/{-}$ & \begin{tikzpicture}[scale=.5] 
\tikzstyle{knode}=[circle,draw=black,thick,inner sep=2pt]
\posnode{(n1) at (-1cm,0cm)}
\negnode{(n2) at (1cm,0cm)}
\draw (n1)--(n2);
\end{tikzpicture} & $\displaystyle \frac{-1}{4(t-1)^2}$  \\\hline \end{tabular} 
\end{table}}

\def\picEEE{\begin{tikzpicture}[scale=0.9]
\tikzstyle{knode}=[circle,draw=black,thick,inner sep=2pt]
\node (n1) at (0:1.5cm) [knode] {};
\node (n2) at (60:1.5cm)  [knode] {};
\node (n2a) at (120:1.5 cm) [knode] {}; 
\node (n3) at (180:1.5cm) [knode] {};
\node (n4) at (240:1.5cm) [knode] {};
\node (n5) at (300:1.5cm) [knode] {};

\path (n1)+(-30:1.2cm) node (n6) [knode] {}
(n3)+(210:1.2cm) node (n7) [knode] {}
(n4)+(-90:1.2cm) node (n8) [knode] {}
(n4)+(210:1.2cm) node (n9) [knode] {}
(n5)+(-30:1.2cm) node (n10) [knode] {}
(n5)+(-90:1.2cm) node (n11) [knode] {};
\draw (n1) to node[midway,sloped,above] {$k+1$}  (n2)
(n2) to node[midway,sloped,above] {$k+1$}  (n2a)
(n2a) to node[midway,sloped,above] {$k+1$} (n3)
(n3) to node[midway,sloped,below] {$1$}  (n4)
(n4) to node[midway,sloped,below] {$1$}  (n5)
(n5) to node[midway,sloped,below] {$1$} (n1)

(n1) to node[sloped,below] {\rotatebox[origin=c]{90}{$k$} } (n6)
(n3) to node[sloped,below] {\rotatebox[origin=c]{-90}{$k$} } (n7)
(n4) to node[sloped,above] {\rotatebox[origin=c]{90}{$k$} }(n8)
(n4) to node[sloped,above] {\rotatebox[origin=c]{-90}{$k$} } (n9)
(n5) to node[sloped,above] {\rotatebox[origin=c]{90}{$k$} }  (n10)
(n5) to node[sloped,above]  {\rotatebox[origin=c]{-90}{$k$} }(n11);
\draw 
(n8) to node[midway,sloped,below] {$k^2$} (n11)
(n7) to node[midway,sloped,below] {$k^2$} (n9);
\end{tikzpicture}
\hfil
\begin{tikzpicture}[scale=0.9]
\tikzstyle{knode}=[circle,draw=black,thick,inner sep=2pt]
\node (n1) at (0:1.5cm) [knode] {};
\node (n2) at (60:1.5cm)  [knode] {};
\node (n2a) at (120:1.5 cm) [knode] {}; 
\node (n3) at (180:1.5cm) [knode] {};
\node (n4) at (240:1.5cm) [knode] {};
\node (n5) at (300:1.5cm) [knode] {};

\path (n1)+(-30:1.2cm) node (n6) [knode] {}
(n3)+(210:1.2cm) node (n7) [knode] {}
(n4)+(-90:1.2cm) node (n8) [knode] {}
(n4)+(210:1.2cm) node (n9) [knode] {}
(n5)+(-30:1.2cm) node (n10) [knode] {}
(n5)+(-90:1.2cm) node (n11) [knode] {};
\draw (n1) to node[midway,sloped,above] {$k+1$}  (n2)
(n2) to node[midway,sloped,above] {$k+1$}  (n2a)
(n2a) to node[midway,sloped,above] {$k+1$} (n3)
(n3) to node[midway,sloped,below] {$1$}  (n4)
(n4) to node[midway,sloped,below] {$1$}  (n5)
(n5) to node[midway,sloped,below] {$1$} (n1)

(n1) to node[sloped,below] {\rotatebox[origin=c]{90}{$k$} } (n6)
(n3) to node[sloped,below] {\rotatebox[origin=c]{-90}{$k$} } (n7)
(n4) to node[sloped,above] {\rotatebox[origin=c]{90}{$k$} }(n8)
(n4) to node[sloped,above] {\rotatebox[origin=c]{-90}{$k$} } (n9)
(n5) to node[sloped,above] {\rotatebox[origin=c]{90}{$k$} }  (n10)
(n5) to node[sloped,above]  {\rotatebox[origin=c]{-90}{$k$} }(n11);
\draw (n6) to node[midway,sloped,below] {$k^2$} (n10);
\draw[white] (n8) to node[midway,sloped,below] {$k^2$} (n11);
\end{tikzpicture}

\bigskip

\begin{tikzpicture}[scale=0.9]
\tikzstyle{knode}=[circle,draw=black,thick,inner sep=2pt]
\node (n1) at (0:1.5cm) [knode] {};
\node (n2) at (60:1.5cm)  [knode] {};
\node (n2a) at (120:1.5 cm) [knode] {}; 
\node (n3) at (180:1.5cm) [knode] {};
\node (n4) at (240:1.5cm) [knode] {};
\node (n5) at (300:1.5cm) [knode] {};
\node (n6) at (90:1.86cm) {$\times k+1$};

\path (n1)+(-30:1.2cm) node (n6) [knode] {$k$}
(n3)+(210:1.2cm) node (n7) [knode] {$k$}
(n4)+(-90:1.2cm) node (n8) [knode] {$k$}
(n4)+(210:1.2cm) node (n9) [knode] {$k$}
(n5)+(-30:1.2cm) node (n10) [knode] {$k$}
(n5)+(-90:1.2cm) node (n11) [knode] {$k$};
\draw[ultra thick] (n1)--(n2) 
(n2a)-- (n3);
\draw (n2)--(n2a)
(n3)--(n4)--(n5)--(n1);

\draw[rounded corners,dashed] (-1 cm,1.549cm) rectangle (1 cm,1.049 cm);

\draw[ultra thick](n1)--(n6)
(n3) --(n7)
(n4)--(n8)
(n4)--(n9)
(n5)--(n10)
(n5)--(n11);
\draw[ultra thick] 
(n8)--(n11)
(n7)--(n9);
\end{tikzpicture}
\hfil
\begin{tikzpicture}[scale=0.9]
\tikzstyle{knode}=[circle,draw=black,thick,inner sep=2pt]
\node (n1) at (0:1.5cm) [knode] {};
\node (n2) at (60:1.5cm)  [knode] {};
\node (n2a) at (120:1.5 cm) [knode] {}; 
\node (n3) at (180:1.5cm) [knode] {};
\node (n4) at (240:1.5cm) [knode] {};
\node (n5) at (300:1.5cm) [knode] {};
\node (n6) at (90:1.8cm) {$\times k+1$};

\path (n1)+(-30:1.2cm) node (n6) [knode] {$k$}
(n3)+(210:1.2cm) node (n7) [knode] {$k$}
(n4)+(-90:1.2cm) node (n8) [knode] {$k$}
(n4)+(210:1.2cm) node (n9) [knode] {$k$}
(n5)+(-30:1.2cm) node (n10) [knode] {$k$}
(n5)+(-90:1.2cm) node (n11) [knode] {$k$};
\draw[ultra thick] (n1)--(n2) 
(n2a)-- (n3);
\draw (n2)--(n2a)
(n3)--(n4)--(n5)--(n1);

\draw[rounded corners,dashed] (-1 cm,1.549cm) rectangle (1 cm,1.049 cm);

\draw[ultra thick](n1)--(n6)
(n3) --(n7)
(n4)--(n8)
(n4)--(n9)
(n5)--(n10)
(n5)--(n11);
\draw[ultra thick] 
(n6)--(n10);
\draw[ultra thick,white]
(n8)--(n11);
\end{tikzpicture}}

\def\picGGG{\begin{tikzpicture}[scale=0.8]
\tikzstyle{knode}=[circle,draw=black,thick,inner sep=2pt]
\node (n1) at (90:1.5cm) [knode] {};
\node (n2) at (210:1.5cm)  [knode] {};
\node (n3) at (330:1.5 cm) [knode] {}; 

\path (n1)+(150:1.2cm) node (n4) [knode] {}
(n2)+(150:1.2cm) node (n5) [knode] {}
(n3)+(30:1.2cm) node (n6) [knode] {}
(n1)+(30:1.2cm) node (n7) [knode] {};
\draw (n2) to node[midway,sloped,below] {$k+1$}  (n3)
(n1) to node[midway,sloped,above] {$1$}  (n2)
(n3) to node[midway,sloped,above] {$1$} (n1)

(n1) to node[sloped,below] {\rotatebox[origin=c]{90}{$k$} } (n4)
(n2) to node[sloped,above] {\rotatebox[origin=c]{90}{$k$} } (n5)
(n3) to node[sloped,above] {\rotatebox[origin=c]{-90}{$k$} }(n6)
(n1) to node[sloped,below] {\rotatebox[origin=c]{-90}{$k$} } (n7);
\draw 
(n6) to node[midway,sloped,above] {$k^2$} (n7)
(n4) to node[midway,sloped,above] {$k^2$} (n5);
\end{tikzpicture}
\hfil
\begin{tikzpicture}[scale=0.8]
\tikzstyle{knode}=[circle,draw=black,thick,inner sep=2pt]
\node (n1) at (90:1.5cm) [knode] {1};
\node (n2) at (210:1.5cm)  [knode] {2};
\node (n3) at (330:1.5 cm) [knode] {2}; 

\path (n1)+(150:1.2cm) node (n4) [knode] {2}
(n2)+(150:1.2cm) node (n5) [knode] {1}
(n3)+(30:1.2cm) node (n6) [knode] {1}
(n1)+(30:1.2cm) node (n7) [knode] {2};
\draw[ultra thick] (n1)--(n2)--(n3)--(n1)

(n1)--(n4)
(n2)--(n5)
(n3)--(n6)
(n1)--(n7);
\draw[ultra thick] 
(n6)--(n7)
(n4)--(n5);
\end{tikzpicture}
}

\renewcommand{\phi}{\varphi}

\theoremstyle{plain}

\newtheorem{theorem}{Theorem}

\newtheorem{lemma}{Lemma}
\newtheorem{proposition}{Proposition}

\theoremstyle{definition}
\newtheorem{definition}{Definition}

\DeclareMathOperator{\sgn}{sgn}
\DeclareMathOperator{\trace}{trace}

\allowdisplaybreaks

\begin{document}
\title{A cospectral family of graphs for the normalized~Laplacian found by toggling}
\author{Steve Butler\thanks{Dept.\ of Mathematics, Iowa State University, Ames, IA 50011, USA\newline ({\tt \{butler,keheysse\}@iastate.edu})}~\thanks{Partially supported by an NSA Young Investigator Grant} \and Kristin Heysse\footnotemark[1]}

\maketitle



\begin{abstract}
We give a construction of a family of (weighted) graphs that are pairwise cospectral with respect to the normalized Laplacian matrix, or equivalently probability transition matrix.  This construction can be used to form pairs of cospectral graphs with differing number of edges, including situations where one graph is a subgraph of the other.  The method used to demonstrate cospectrality is by showing the characteristic polynomials are equal.
\end{abstract}



\section{Introduction}\label{sec:introduction}
Spectral graph theory studies the relationship between the structure of a graph and the eigenvalues of a particular matrix associated with that graph. There are several matrices that are commonly studied, each with merits and limitations. These limitations exist because graphs can be constructed which have the same spectrum with respect to the matrix and are fundamentally different in some structural aspect. Such graphs are called \emph{cospectral}.

There are many possible matrices to consider, and the matrix we consider in this paper is the normalized Laplacian (see \cite{BC,Chung}). The rows and columns of this matrix are indexed by the vertices, and for a simple graph the matrix is defined as follows:
\[ 
\mathcal{L}(i,j)=\left\{\begin{array}{c@{\qquad}l} 
1 & \text{if $i=j$, and vertex $i$ is not isolated};\\[5pt] 
{\displaystyle\frac{-1}{\sqrt{d_i d_j}}} & \text{if } i {\sim} j; \\[5pt]
0 & \text{otherwise;} \end{array} \right. 
\]
where $d_i$ is the degree of vertex $i$.

In this paper we want to look at the more general setting of edge-weighted graphs, i.e., there is a symmetric, non-negative weight function, $w(i,j)$ on the edges.  The degree of a vertex now corresponds to the sum of the weights of the incident edges, i.e., $d_i=\sum_{i{\sim}j}w(i,j)$.  The normalized Laplacian for weighted graphs is defined in the following way:
\[ 
\mathcal{L}(i,j)=\left\{ \begin{array}{c@{\qquad}l} 
1 & \text{if $i=j$, and vertex $i$ is not isolated};\\[5pt] {\displaystyle\frac{-w(i,j)}{\sqrt{d_i d_j}}} & \text{if }i {\sim} j;\\[5pt] 
0 & \text{otherwise.}  \end{array} \right.
\] 
(A simple graph corresponds to the case where $w(i,j)\in\{0,1\}$ for all $i,j$.)  We note that when the graph has no isolated vertices, $\mathcal{L}$ can be written as $\mathcal{L}=D^{-1/2}(D-A)D^{-1/2}$, where $A_{i,j}=w(i,j)$ and $D$ is the diagonal degree matrix.  Finally, we point out that this matrix is connected with the probability transition matrix $D^{-1}A$ of a random walk.  In particular, two graphs with no isolated vertices are cospectral for $\mathcal{L}$ if and only if they are cospectral for $D^{-1}A$.

There has been some interest in the construction of cospectral graphs for the normalized Laplacian.  Cavers \cite{Cavers} showed that a restricted variation of Godsil-McKay switching (see \cite{GM}) preserves the spectrum, while Butler and Grout \cite{BG} showed that gluing in two different special bipartite graphs into some arbitrary graph resulted in a pair of cospectral graphs.  In both cases, the operation preserved the number of edges in the graph. 

On the other hand, it is possible for graphs with differing number of edges to be cospectral with respect to the normalized Laplacian.  The classic example of this is complete bipartite graphs $K_{p,q}$ which have spectrum $\{0,1^{(p+q-2)},2\}$ (here the exponent is indicating multiplicity).  For example, the (sparse) star $K_{1,2n-1}$ is cospectral with the (dense) regular graph $K_{n,n}$.  Until recently, this was the \emph{only} known construction of cospectral graphs with differing number of edges.  Butler and Grout \cite{BG} gave some examples of small graphs found by exhaustive computation that differ in the number of edges, including some where one graph was a subgraph of the other.  Butler \cite{twins} expanded on this example to form an infinite family and showed how to construct many pairs of bipartite graphs which were cospectral.

In this paper we introduce a new construction of cospectral graphs for the normalized Laplacian which can differ in the number of edges.  The basic idea is to form a ring of linked modules, and then a similar graph where we interchange the role of two of the modules (what we term ``toggling'').  The resulting pair of graphs are cospectral with respect to the normalized Laplacian.  An example of this construction is shown in Figure~\ref{firstexample}.  Note that the left graph is a subgraph of the right graph.

\begin{figure}[htb]
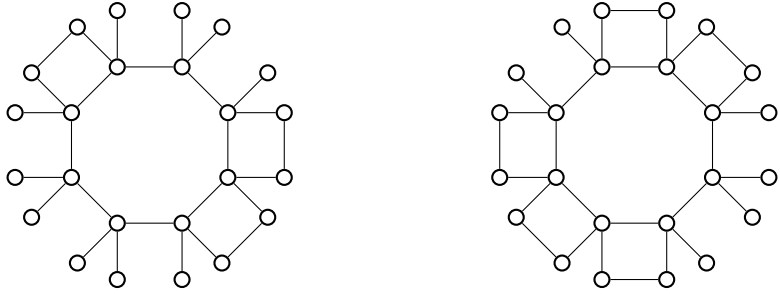

\centering
\picAAA
\caption{A pair of cospectral graphs for $\mathcal{L}$ related by toggling}
\label{firstexample}
\end{figure}

In Section~\ref{Construction}, we give a formal description of this family, of toggling, and state the main result.  In Section~\ref{sec:charpoly}, we show how to compute the characteristic polynomial of the normalized Laplacian by using decompositions.  We then break the decompositions of a graph in our family into those which contain a ``long'' cycle (see Section~\ref{sec:long}) and those which do not (see Section~\ref{sec:short}), and in particular conclude the characteristic polynomials are equal so the graphs must be cospectral.  In Section~\ref{sec:weighted} we show how to go from weighted graphs to simple graphs which are cospectral with respect to the normalized Laplacian.


\section{Construction}\label{Construction}
Our family of graphs are formed as a ring composed of three different types of (weighted) modules: the \emph{path} on four vertices, the \emph{cycle} on four vertices, and the \emph{edge} on two vertices, which we label as {\tt P}, {\tt C}, and  {\tt E}, respectively.  The modules are shown in Figure~\ref{modules} where we have marked the edge weights using a parameter $k$ where $k>0$ is for now arbitrary.  Each module has special vertices marked with ``$+$'' and ``$-$'' which can be thought of as poles of a magnet to indicate how consecutive modules will connect.  In particular, the ``$+$'' vertex on one module will connect with the ``$-$'' vertex on the next module.  We will refer to these two special vertices as the \emph{signed} vertices.

\begin{figure}[htb]
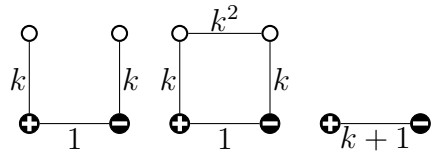

\centering
\picBBB
\caption{The {\tt P}, {\tt C}, and {\tt E} modules, respectively}
\label{modules}
\end{figure}

A graph in our family is formed by connecting $\tau$ modules together in a cycle.  In particular, such graph can be associated with a word using the letters {\tt P}, {\tt C} and {\tt E}.  As an example, starting with the top module and reading clockwise, the two graphs shown in Figure~\ref{firstexample} (where $k=1$) have the words {\tt PPCCPPPC} and {\tt CCPPCCCP}.  Note that given a graph in our family there are many possible words, i.e., we can choose any module to start and any possible direction.  On the other hand, given a word, there is a unique graph.

\begin{definition}
Given a word $W=\ell_1\ell_2\ldots\ell_\tau$ where $\ell_i\in\{{\tt P},{\tt C},{\tt E}\}$ and $\tau\ge3$.  Then $G(W)$ is the graph obtained by connecting the corresponding $\tau$ modules in cyclic order as indicated by the word where consecutive modules connect on the signed vertices, and where the final module will connect to the first module.
\end{definition}

We note that the two words we constructed for the graphs in Figure~\ref{firstexample} are related by interchanging the roles of ${\tt P}$ and ${\tt C}$.  This will generalize as follows.

\begin{definition}
Given a cyclic word $W$ composed of the letters {\tt P}, {\tt C}, and {\tt E}.  Then the \emph{toggling} of $W$ is $W^T$, the word formed by taking $W$ and replacing every {\tt P} by {\tt C} and every {\tt C} by {\tt P}. The occurrences of {\tt E} are unchanged.
\end{definition}

The motivation for the use of the word ``toggling'' is to notice that the difference between $G(W)$ and $G(W^T)$ is adding or removing the edge on a module which goes between the non-signed vertices.   In essence, we are switching the states of these edges.

We can now state our main result.

\begin{theorem}\label{thm:main}
For a word $W$ of length at least three using the letters {\tt P}, {\tt C}, and {\tt E}, $G(W)$ and $G(W^T)$ are cospectral with respect to the normalized Laplacian.
\end{theorem}

We note that if $W$ does not contain the same number of occurrences of {\tt P} and {\tt C}, then the number of edges in $G(W)$ and $G(W^T)$ will differ and are clearly non-isomorphic. Among other things, we can construct cospectral simple graphs which differ by exactly $m$ edges by setting $k=1$ and using a word in $\tt{P}$ and $\tt{C}$ where there are $m$ more instances of $\tt{C}$ than of $\tt{P}$.   There are also some special words $W$ so that $G(W)$ is a subgraph of $G(W^T)$.  One example of this behavior is $W={\tt CC}\ldots{\tt C}$ and $W^T={\tt PP}\ldots{\tt P}$, though others exist (see Figure~\ref{firstexample}). 

\section{Computing the characteristic polynomial}\label{sec:charpoly}
Our approach will involve showing the characteristic polynomials of $G(W)$ and $G(W^T)$ are equal.  We  start by determining how to compute the characteristic polynomial by the use of generalized cycle decompositions (see \cite{Brualdi}). For an $n \times n$ matrix $M=[m_{i,j}]$,
\[
\det(M)=\sum_{\sigma \in S_n} \sgn(\sigma)\underbrace{m_{1,\sigma(1)}m_{2,\sigma(2)}\cdots m_{n,\sigma(n)}}_{:=w_M(\sigma)}=\sum_{\sigma \in S_n} \sgn(\sigma) w_M(\sigma). 
\]

Let $G_M$ denote the digraph which corresponds to $M$, meaning it has $i{\to}j$ if and only if $m_{i,j}\ne0$. We can consider a permutation $\sigma$ which contributes a nonzero term to $\det(M)$. The factors of $w_M(\sigma)$ correspond to $n$ edges such that each vertex has in-degree and out-degree equal to one, as each vertex will appear as the first and second index somewhere in $w_M(\sigma)$.  Such a collection of edges is a \emph{generalized cycle decomposition} of $G_M$.  There are three possible structures in a generalized cycle decomposition: loops (a directed edge that goes into and out of the same vertex), edges (pairs of directed edges $i{\to}j$ and $j{\to}i$), and longer directed cycles.  More generally, if we think of loops and edges as cycles of length one and two, respectively, then a generalized cycle decomposition is a collection of disjoint cycles so that every vertex is in exactly one cycle.

In the case when the matrix $M$ is symmetric, many of these generalized cycle decompositions will contribute the same factor to the determinant. For example, changing the orientation on a long cycle gives a different decomposition but does not change $\sgn(\sigma)w_M(\sigma)$.  With this in mind we consider decompositions.

\begin{definition}
Let $G$ be an undirected (weighted) graph.  Then a \emph{decomposition}, $D$, is a subgraph consisting of disjoint edges and cycles.
\end{definition}

When $M$ is symmetric, we can treat $G_M$ as an undirected graph.  Every generalized cycle decomposition now corresponds to a unique decomposition, $D$, by removing loops and dropping the orientation on the long cycles.  Conversely, if we let $s=s(D)$ denote the number of cycles of length at least three in the decomposition $D$, then each decomposition corresponds to a collection of $2^s$ different generalized cycle decompositions.  Namely, any vertex not in an edge or a cycle has a loop added, edges become cycles of length two, and each of the $s$ cycles of length at least $3$ have one of two possible orientations chosen. 

If we let $e(D)$ count the number of cycles in the decomposition which have an even number of vertices (including edges), and $F(D)$ be the set of isolated edges in the decomposition $D$, then we have the following result.

\begin{proposition}\label{prop:charpoly}
Let $G$ be a weighted graph on $n$ vertices without loops or isolated vertices.  Then the characteristic polynomial of the normalized Laplacian matrix is
\[
p(t)=\sum_{D}(-1)^{e(D)}2^{s(D)}(t-1)^{n-|V(D)|}\frac{ \prod_{\{i,j\}\in E(D)}w(i,j)\prod_{\{i,j\}\in F(D)}w(i,j)}{\prod_{i\in V(D)}d_i}
\]
where the sum runs over all decompositions $D$ of the graph $G$.
\end{proposition}
\begin{proof}
The characteristic polynomial with respect to the normalized Laplacian can be written as
\begin{align*}
p(t) &=\det(tI-\mathcal{L})\\
&=\det\big(tI-D^{-1/2}(D-A)D^{-1/2}\big)\\
&=\det\big(\underbrace{(t-1)I+D^{-1/2}AD^{-1/2}}_{=M}\big).
\end{align*}
The graph $G_M$ (ignoring loops) has the same edges and non-edges as $G$, and so we can use decompositions to compute the determinant.

In particular, every decomposition of $G$ will relate to $2^{s(D)}$ generalized cycle decompositions.  For each such generalized cycle decomposition corresponding to a permutation $\sigma$, we have $\sgn(\sigma)=(-1)^{e(D)}$.  We will have $n-|V(D)|$ loops which each contribute $(t-1)$.  The non-loop edges $i{\to}j$ will contribute $w(i,j)/\sqrt{d_id_j}$.  Now we recall that each vertex in a generalized cycle decomposition has one edge coming in and one edge going out, and therefore for each vertex $i$ in $V(D)$ we will have $\sqrt{d_i}$ occurring twice in the denominator giving us the $d_i$.  Finally, for cycles of length three or greater we only use each edge once in the generalized cycle decomposition, but for cycles of length two we use the same edge for both directions and so we use the edge twice.
\end{proof}

\section{Decompositions of $G(W)$ with a long cycle}\label{sec:long}
Proposition~\ref{prop:charpoly} shows that we can determine the characteristic polynomial by looking at decompositions of the graph.  In this section we will consider the collection of decompositions of a graph $G(W)$ which contain a long cycle, i.e., a cycle which passes through all of the signed vertices in $G(W)$.  We denote the set of these decompositions as $L$.

\begin{lemma}\label{lem:long}
Let $W$ be a word of length $\tau$ with $\ell$ occurrences of {\tt P} and $m$ occurrences of {\tt C}.  Then for $G(W)$ we have
\begin{multline*}
\sum_{D\in L}(-1)^{e(D)}2^{s(D)}(t-1)^{n-|V(D)|}\frac{ \prod_{\{i,j\}\in E(D)}w(i,j)\prod_{\{i,j\}\in F(D)}w(i,j)}{\prod_{i\in V(D)}d_i}\\
=\frac{(-1)^{\tau-1}(t-1)^{2(m+\ell)}}{2^{\tau-1}(k+1)^{m+\ell}}
\end{multline*}
\end{lemma}

\begin{proof}
Knowing we have a long cycle yields a lot of information about the decomposition $D$ in $G(W)$.  In particular, for a module of type ${\tt P}$ or ${ \tt E}$, the decomposition will contain only the edge between the signed vertices.  These are shown in Figure~\ref{xzlongdecomp}, where edge weights have been removed for clarity.

\begin{figure}[htb]
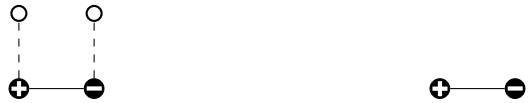

\centering
\picCCC
\caption{Forced decomposition for {\tt P} and {\tt E}, respectively}
\label{xzlongdecomp}
\end{figure}

For a module of type {\tt C} the situation is a more interesting as there are three different options for the decomposition.  Namely, that the long cycle passes only through the signed vertices; the long cycle passes through the signed vertices and there is an edge between the unsigned vertices; the long cycle passes through all of the vertices.  These three possibilities are shown in Figure~\ref{ylongdecomp}.

\begin{figure}[htb]
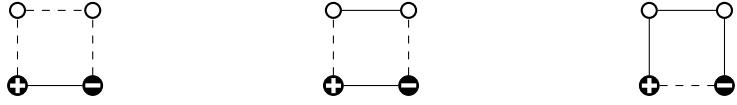

\centering
\picDDD
\caption{Three possible decompositions for {\tt C} in a long cycled decomposition}
\label{ylongdecomp}
\end{figure}

Now suppose that among the $m$ modules of type {\tt C} that precisely $h$ of them are the configuration shown on the left in Figure~\ref{ylongdecomp}; $i$ of them are the configuration shown in the center in Figure~\ref{ylongdecomp}; and $j$ of them are the configuration shown on the right in Figure~\ref{ylongdecomp}.  The choices of which {\tt C} modules behave in which way is arbitrary.  Summing over all the possibilities gives the following.
\begin{multline*}
\sum_{D\in L}(-1)^{e(D)}2^{s(D)}(t-1)^{n-|V(D)|}\frac{ \prod_{\{i,j\}\in E(D)}w(i,j)\prod_{\{i,j\}\in F(D)}w(i,j)}{\prod_{i\in V(D)}d_i}\\
=\sum_{h+i+j=m}\frac{2(-1)^{\tau-1}(k+1)^{\tau-\ell-m}(t-1)^{2\ell}}{\big(2(k+1)\big)^\tau}\times\\
{m\choose h,i,j}\big((t-1)^2\big)^h\bigg({(-1)k^4\over\big(k(k+1)\big)^2}\bigg)^i\bigg({k^4\over\big(k(k+1)\big)^2}\bigg)^j
\end{multline*}


We have $2^{s(D)}=2$ because there is only one cycle of length greater than three, namely the long cycle which contains all the signed vertices. The $e(D)$ will count the number of $\tt{C}$ modules in the middle configuration and possibly the long cycle itself. Regardless of the number of $\tt{C}$ modules in the configuration on the right, the contribution from the long cycle to $(-1)^{e(D)}$ will be $(-1)^{\tau-1}$. Consider first the contributions of isolated vertices and edge weights of the $\tt{P}$ and $\tt{E}$ modules. The $(k+1)^{\tau-\ell-m}$ is the weight of the edges on the long cycle coming from the modules of type {\tt E}.  The $(t-1)^{2\ell}$ accounts for the isolated vertices from the modules of type {\tt P}. Further, the $\big(2(k+1)\big)^\tau$ is the product of the degrees of the signed vertices (each such vertex has degree $2(k+1)$ as can be seen by noting that in the modules the signed vertices have degree $k+1$ and then we identify two such vertices). 

It remains to account for the portions of the decomposotions formed on the $\tt{C}$ modules which are not the signed vertices. The ${m\choose h,i,j}=m!/(h!i!j!)$ is the multinomial coefficient for how many ways to choose the different module configurations for {\tt C}, and the final three factors are the contributions from each configuration formed by accounting for isolated vertices, edge weights, and the degrees of vertices in the decomposition (i.e., a pair of isolated vertices or the product of the edge weights over product of degrees). Notice that for the middle configuration, the contribution to $(-1)^{e(D)}$ has been appropriately grouped. 

Now we can simplify by pulling out the terms which do not depend on the sum and cancelling.  For the terms in the sum we can use the multinomial theorem to simplify. Continuing the above computation, we now have
\begin{align*}
&=\frac{(-1)^{\tau-1}(t-1)^{2\ell}}{2^{\tau-1}(k+1)^{\ell+m}}\bigg((t-1)^2-\frac{k^4}{\big(k(k+1)^2\big)^2}+\frac{k^4}{\big(k(k+1)^2\big)^2}\bigg)^m\\
&=\frac{(-1)^{\tau-1}(t-1)^{2\ell}}{2^{\tau-1}(k+1)^{\ell+m}}(t-1)^{2m}=\frac{(-1)^{\tau-1}(t-1)^{2(m+\ell)}}{2^{\tau-1}(k+1)^{m+\ell}}.\qedhere
\end{align*}
\end{proof}

The important thing to note is that the expression in Lemma~\ref{lem:long} will be the same for $W$ and $W^T$ because $m+\ell$ is invariant under toggling.

\section{Decompositions of $G(W)$ without a long cycle}\label{sec:short}
Any cycle in a decomposition with edges in consecutive modules would have to go through all of the modules to close up.  In particular, if there is not a long cycle in our decomposition $D$, then the decomposition is composed of only edges and $C_4$'s which lie in individual modules.

We consider what decompositions can happen in a single module and how decompositions in consecutive modules interact.  The first task is straightforward to carry out, and in Tables~\ref{table:P}, \ref{table:C}, and \ref{table:E} we show the possible \emph{local} decompositions for each module.  To help facilitate the analysis we have grouped the local decompositions by which signed vertices (if any) are used.

\tableP

\tableC

\tableE

The next part is to understand the transitions between modules, i.e., how local decompositions interact.  We have already grouped the local decompositions by which of the signed vertices are used.  We now note that if signed vertices are used in by a local decompositon in one module, it influences which of the signed vertices are available for use in the next module.  This is indicated by the following transition matrix with rows and columns indexed by subsets of the signed vertices:
\[
Q=\bordermatrix{
&~\emptyset~&~{+}~&~{-}~&{+}/{-}\cr
\emptyset&1&1&1&1\cr
{+}&1&1&1&1\cr
{-}&1&0&1&0\cr
{+}/{-}&1&0&1&0}.
\]
Using $Q$ we can now count the number of ways that we can have decompositions use the signed vertices in the modules for $G(W)$.  This is done using the transfer matrix method (see \cite{GS}), and in particular is equal to the number of closed walks in the directed graph corresponding to $Q$ which have the same length as the length of the word. 
We need to go one step further and for every module add the contribution of the local decomposition.

This final part is done by adding in diagonal weight matrices where the diagonal entries correspond to the contribution of the decomposition for that particular module.  These contributions are found by $(-1)$ raised to the number of even cycles (i.e., edges or $C_4$'s) times the product of the edge weights used in the local decomposition (remembering for an edge to use that edge twice), divided by the product of the degrees of any vertex used in the decomposition.  The only subtle part is handling the vertices which will not be a part of a decomposition in any module.  What we do is assume at the beginning that \emph{every} vertex is isolated and contributes a $(t-1)$ then whenever a vertex becomes a part of the decomposition we divide by $(t-1)$ to correct (the choice of this approach is because signed vertices lie in two modules, hence while it might not be in the decomposition of one module it could be in the decomposition of the other).  When there are several possible decompositions in a given case we add them together to form the entry for the weight matrix. 
The contributions were previously listed in the tables and become the diagonal entries of the weight matrices.  We therefore have the following weight matrices.
\begin{align*}
X_{\tt P}&=\left(\begin{array}{cccc}
1 & 0 & 0 & 0 \\
0 & \frac{-k}{(t-1)^2(2k+2)} & 0 & 0 \\
0 & 0 & \frac{-k}{(t-1)^2(2k+2)} & 0 \\
0 & 0 & 0 & \frac{k^2-(t-1)^2}{(t-1)^4(2k+2)^2}
\end{array}\right)
\\
X_{\tt C}&=\left(\begin{array}{cccc}
1-\frac{k^2}{(t-1)^2(k+1)^2}  & 0 & 0 & 0 \\
0 & \frac{-k}{2(t-1)^2(k+1)^2} & 0 & 0 \\
0 & 0 & \frac{-k}{2(t-1)^2(k+1)^2} & 0 \\
0 & 0 & 0 & \frac{-1}{4(t-1)^2(k+1)^2}
\end{array}\right) 
\\
X_{\tt E}&=\left(\begin{array}{cccc}
1 & 0 & 0 & 0 \\
0 & 0 & 0 & 0 \\
0 & 0 & 0 & 0 \\
0 & 0 & 0 & \frac{-1}{4(t-1)^2}
\end{array}\right)
\end{align*}

So for the graph $G(\ell_1\ell_2\cdots \ell_\tau)$, we have the following:
\begin{multline}\label{eq:short}
\sum_{D\notin L}(-1)^{e(D)}2^{s(D)}(t-1)^{n-|V(D)|}\frac{ \prod_{\{i,j\}\in E(D)}w(i,j)\prod_{\{i,j\}\in F(D)}w(i,j)}{\prod_{i\in V(D)}d_i}\\
=(t-1)^{|V(G(W))|}\trace(QX_{\ell_1}QX_{\ell_2}\cdots QX_{\ell_\tau}).
\end{multline}

We now focus on rewriting the trace expression in \eqref{eq:short}.  To start we note that $Q=RSR^{-1}$ where
\[
R = \left(\begin{array}{cccc}
1&-1&1&1\\
1&-1&0&0\\
\frac12&1&0&-1\\
\frac12&1&-2&0
\end{array}\right),
\quad\text{and}\quad
S = \left(\begin{array}{cccc}
3&0&0&0\\
0&0&1&0\\
0&0&0&0\\
0&0&0&0
\end{array}\right).
\]
Combining this with $\trace(AB)=\trace(BA)$ we can conclude
\begin{multline*}
(t-1)^{|V(G(W))|}\trace(QX_{\ell_1}QX_{\ell_2}\cdots QX_{\ell_\tau})\\
=(t-1)^{|V(G(W))|}\trace(RSR^{-1}X_{\ell_1}RSR^{-1}X_{\ell_2}\cdots RSR^{-1}X_{\ell_\tau})\\
=(t-1)^{|V(G(W))|}\trace\big((SR^{-1}X_{\ell_1}R)(SR^{-1}X_{\ell_2}R)\cdots (SR^{-1}X_{\ell_\tau}R)\big).
\end{multline*}
Because $S$ has two rows of $0$'s this simplifies the matrices that we have to deal with.  In particular we have
\begin{align*}
SR^{-1}X_{\tt P}R&=\left(\begin{array}{cc}Y_{\tt P}&Z_{\tt P}\\O&O\end{array}\right),\\
SR^{-1}X_{\tt C}R&=\left(\begin{array}{cc}Y_{\tt C}&Z_{\tt C}\\O&O\end{array}\right),\text{ and}\\
SR^{-1}X_{\tt E}R&=\left(\begin{array}{cc}Y_{\tt E}&Z_{\tt E}\\O&O\end{array}\right),
\end{align*}
where if we let $u:=t-1$ then
\begin{align*}
Y_{\tt P}&=\left(\begin{array}{cc}\frac{16k^2u^4 + 32ku^4 - 8k^2u^2 + 16u^4 - 8ku^2 + k^2 - u^2}{12(k+1)^2u^4}&\frac{-8k^2u^4 - 16ku^4 - 2k^2u^2 - 8u^4 - 2ku^2 + k^2 - u^2}{6(k+1)^2u^4}\\[5pt]
\frac{8k^2u^4 + 16ku^4 + 2k^2u^2 + 8u^4 + 2ku^2 - k^2 + u^2}{24(k+1)^2u^4}&\frac{16k^2u^4 + 32ku^4 - 8k^2u^2 + 16u^4 - 8ku^2 + k^2 - u^2}{12(k + 1)^2u^4
}\end{array}\right)\\
Y_{\tt C}&=\left(\begin{array}{cc}
\frac{16k^2u^2 + 32ku^2 - 16k^2 + 16u^2 - 8*k - 1}{12(k+1)^2u^2}& \frac{-8k^2u^2 - 16ku^2 + 8k^2 - 8u^2 - 2k - 1}{6(k+1)^2u^2}\\[5pt]
\frac{8k^2u^2 + 16ku^2 - 8k^2 + 8u^2 + 2k + 1}{24(k+1)^2u^2}&\frac{-4k^2u^2 - 8ku^2 + 4k^2 - 4u^2 - 4k + 1}{12(k+1)^2u^2}
\end{array}\right)\\
Y_{\tt E}&=\left(\begin{array}{cc}\frac{16u^2-1}{12u^2}&\frac{-8u^2-1}{6u^2}\\[5pt]\frac{8u^2+1}{24u^2}&\frac{-4u^2+1}{12u^2}\end{array}\right)
\end{align*}

Because we can carry out block matrix multiplication, we note that the resulting upper left block will be the product of the upper left blocks and that the resulting lower right block will be the all zeroes matrix.  This allows us to conclude the following:
\begin{equation*}
(t-1)^{|V(G(W))|}\trace(QX_{\ell_1}QX_{\ell_2}\cdots QX_{\ell_\tau})
=(t-1)^{|V(G(W))|}\trace(Y_{\ell_1}Y_{\ell_2}\cdots Y_{\ell_\tau})
\end{equation*}

There is no convenient way to find a simple expression for these decompositions as we did for the long cycles.  However, it suffices to show that the toggled words will produce equivalent results, which is what we now show. 

\begin{lemma}\label{lem:short}
Let $W=\ell_1\ell_2\ldots \ell_\tau$ and $W^T=\gamma_1\gamma_2\ldots\gamma_\tau$.  Then
\[
(t-1)^{|V(G(W))|}\trace(Y_{\ell_1}Y_{\ell_2}\cdots Y_{\ell_\tau}) = 
(t-1)^{|V(G(W^T))|}\trace(Y_{\gamma_1}Y_{\gamma_2}\cdots Y_{\gamma_\tau}).
\]
\end{lemma}
\begin{proof}
Both sides are polynomials, and so it suffices to verify that the relationship holds for $t\ne0,1,2$ (i.e., if two polynomials agree at all but three points, they must agree everywhere).  To show that they are equal, we will make use of the following special matrix,
\[
U=\left(\begin{array}{cc}
20u^2-2&-32u^2-4\\
8u^2+1&-20u^2+2
\end{array}\right).
\]
This matrix has the following special properties, which can be verified by carrying out matrix multiplication:
\begin{itemize}
\item $UY_{\tt P}=Y_{\tt C}U$.
\item $UY_{\tt C}=Y_{\tt P}U$.
\item $UY_{\tt E}=Y_{\tt E}U$.
\end{itemize}
These properties are key, in that they indicate we can pass $U$ through one of the $Y_{*}$ matrices but we need to change the matrix in the same way that we do in the toggling operation.

For $t\ne 0,1,2$ we have that $U$ is invertible and so by repeated application of the above properties we have
\begin{align*}
(t-1)^{|V(G(W))|}\trace(Y_{\ell_1}Y_{\ell_2}\cdots Y_{\ell_\tau})
&=(t-1)^{|V(G(W))|}\trace(UY_{\ell_1}Y_{\ell_2}\cdots Y_{\ell_\tau}U^{-1})\\
&=(t-1)^{|V(G(W))|}\trace(Y_{\gamma_1}UY_{\ell_2}\cdots Y_{\ell_\tau}U^{-1})\\
&=(t-1)^{|V(G(W))|}\trace(Y_{\gamma_1}Y_{\gamma_2}U\cdots Y_{\ell_\tau}U^{-1})\\
&=\cdots\\
&=(t-1)^{|V(G(W))|}\trace(Y_{\gamma_1}Y_{\gamma_2}\cdots U Y_{\ell_\tau}U^{-1})\\
&=(t-1)^{|V(G(W))|}\trace(Y_{\gamma_1}Y_{\gamma_2}\cdots Y_{\gamma_\tau}UU^{-1})\\
&=(t-1)^{|V(G(W))|}\trace(Y_{\gamma_1}Y_{\gamma_2}\cdots Y_{\gamma_\tau})\\
&=(t-1)^{|V(G(W^T))|}\trace(Y_{\gamma_1}Y_{\gamma_2}\cdots Y_{\gamma_\tau}),
\end{align*}
where in the last we use that toggling does not change the number of vertices in the graph.
\end{proof}

\begin{proof}[Proof of Theorem~\ref{thm:main}]
To show that the graphs $G(W)$ and $G(W^T)$ are cospectral we can show that they have the same characteristic polynomial.  We use Proposition~\ref{prop:charpoly} and consider all the possible decompositions.  Lemma~\ref{lem:long} shows that the sum over all the decompositions which contain a long cycle are equal while Lemma~\ref{lem:short} shows that the sum over all the decompositions which do not contain a long cycle are also equal.  Thus the sum over all decompositions is equal, and the theorem is established.
\end{proof}

\section{Weighted Graphs to Simple Graphs}\label{sec:weighted}
We have considered graphs with edge weights in terms of a parameter $k$ as shown in Figure~\ref{modules}.  By letting $k=1$ and restricting to {\tt P} and {\tt C} modules we will produce cospectral simple graphs.

Simple graphs can also be obtained by appropriately ``blowing up'' our graph.  This works by replacing vertices by independent sets.  An edge between $u$ and $v$ which has been replaced by $r$ and $s$ vertices respectively then becomes a complete bipartite graphs between the two independent sets with all edge weights $w(u,v)/rs$.  (Note that $r$ and $s$ are generally chosen so that this new edge weight is $1$, i.e., so the new graph is a simple graph.)  Similarly several consecutive {\tt E} edges with weight $k+1$ can become $k+1$ parallel paths.  A discussion on how eigenvalues for the normalized Laplacian work for blowups can be found in \cite{twins}.  In particular, it is known that the eigenvalues of the blowups are determined from the eigenvalues of the original graphs (which we have shown to be cospectral) and the remaining eigenvalues will come from the blowup procedure, which will be the same for both graphs.

As an example in Figure~\ref{blowup} we start with the cospectral graphs 
$\tt{EEEPCC}$ and $\tt{EEECPP}$.  This figure also contains the blowups which result by replacing the unsigned vertices in {\tt C} and {\tt P} modules with $k$ independent vertices (marked by putting $k$ inside the vertex and making the lines bold to represent complete bipartite graphs), the three consecutive {\tt E} edges become $k+1$ parallel paths of length three.  In particular, the resulting blowups are simple graphs which are also cospectral.

\begin{figure}[htb]
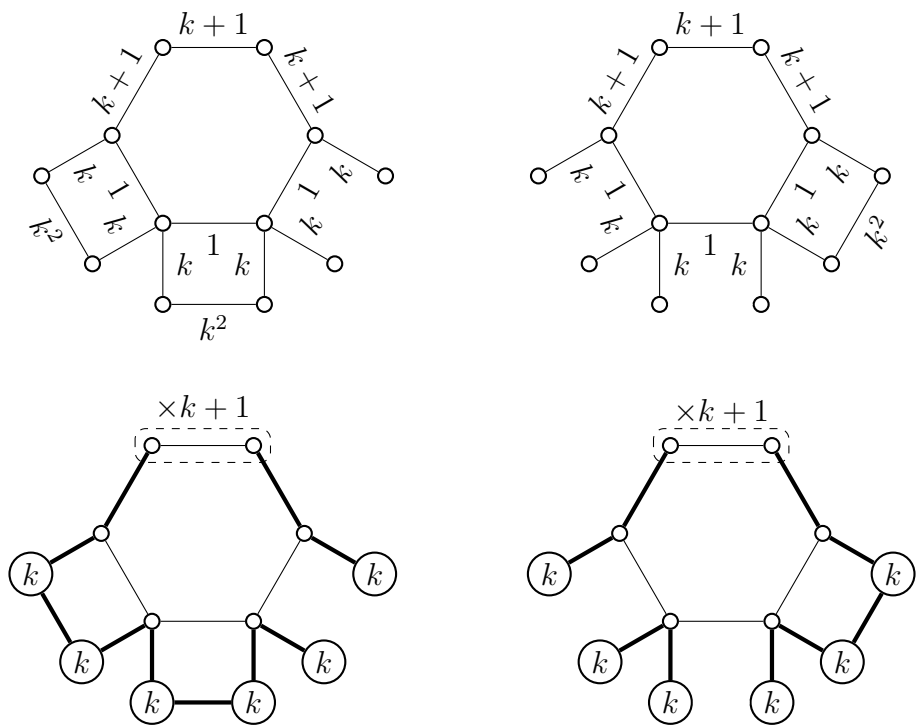

\centering
\picEEE
\caption{The graphs $\tt{EEEPCC}$ and $\tt{EEECPP}$ (above) and their respective blowups (below).}
\label{blowup}
\end{figure}

There are other possibilities.  For instance, from the definition of the normalized Laplacian we note that the matrix does not change if we scale all edge weights by a fixed amount.  So we can first scale the edge weights and then perform a blowup.  A partial example of this is shown in Figure~\ref{singleE} where we consider the graph corresponding to {\tt ECC}.  By setting $k=1$ and then scaling all edge weights by $2$ we get a weighted graph which has as a blowup the graph shown on the right in Figure~\ref{singleE}.  By a similar process we could also do the same for {\tt EPP} to construct a cospectral pair of simple graphs.

\begin{figure}[htb]
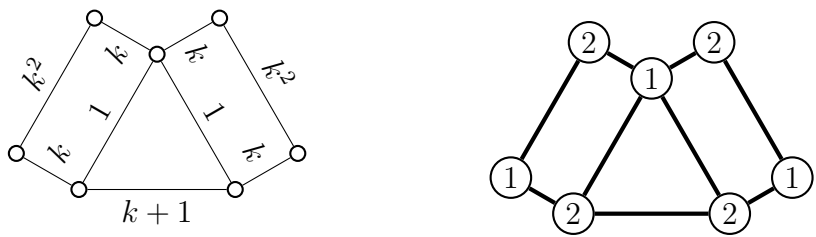

\centering
\picGGG
\caption{The graph $\tt{ECC}$ and corresponding blowup, where $k=1$ and all edge weights have been scaled by $2$.}
\label{singleE}
\end{figure}

\section{Conclusion}
Many, if not most, approaches to establish cospectrality rely on showing that a small perturbation in the graph corresponds to a small, controllable perturbation in the eigenvectors and hence eigenvalues are preserved.  This was \emph{not} the case in this construction, which is why we considered the characteristic polynomials.  Also, while we show that the characteristic polynomials are equal, we never explicitly computed one.  Instead, we showed that the method to determine these polynomials will produce the same answer for a pair of cospectral graphs.  It would be interesting to find additional families where this can occur.

We have been able to establish a large family of cospectral graphs for the normalized Laplacian (and hence also probability transition matrix) which have unusual properties, including cospectral graphs with differing number of edges and graphs cospectral with subgraphs.  There is a vast amount about the spectrum of the normalized Laplacian that is not well understood.  We hope to see more of this area explored in future work.

\bigskip

\noindent\textbf{Acknowledgment.}  The work on this paper was partially conducted while the authors were visiting the Institute for Mathematics and its Applications.

\end{document}